\documentclass{amsart}

\newtheorem{theorem}{Theorem}[section]
\newtheorem{lemma}[theorem]{Lemma}
\newtheorem{cor}{Corollary}
\newtheorem{proposition}[theorem]{Proposition}

\theoremstyle{definition}
\newtheorem{definition}[theorem]{Definition}

\theoremstyle{remark}
\newtheorem{remark}[theorem]{Remark}

\newcommand{\comp}{\mathbb{C}^n}
\newcommand{\mpp}{M^{\perp}}
\newcommand{\npp}{N^{\perp}}

\newcommand{\conv}{\operatorname{Conv}}

\newcommand*\colvec[3][]{
    \begin{bmatrix}\ifx\relax#1\relax\else#1\\\fi#2\\#3\end{bmatrix}
}

\title{Jordan Plane and Numerical Range of Operators involving Two Projections}
\numberwithin{equation}{section}

\begin{document}

\title{Jordan Plane and Numerical Range of Operators involving Two Projections}

%    Information for first author
\author{Jaedeok Kim}
%    Address of record for the research reported here
\address{Department of Mathematical, Computing, and Information Sciences, Jacksonville State University, Jacksonville, Alabama 36265}
%    Current address
%\curraddr{Department of Mathematics and Statistics,
%Case Western Reserve University, Cleveland, Ohio 43403}
\email{jkim@jsu.edu}
%    \thanks will become a 1st page footnote.
%\thanks{The first author was supported in part by NSF Grant \#000000.}

%    Information for second author
\author{Youngmi Kim}
\address{Department of Mathematical, Computing, and Information Sciences, Jacksonville State University, Jacksonville, Alabama 36265}
\email{ykim@jsu.edu}
%\thanks{Support information for the second author.}

%    General info

\date{July 8, 2018 }

%\dedicatory{This paper is dedicated to our advisors.}


\begin{abstract}
We use principal angles between two subspaces to define Jordan planes. Jordan planes provide an optimal way to decompose $\comp$ in relation to given two subspaces. We apply Jordan planes to show that two pairs of of subspaces $(M,N)$ and $(M^{\perp},N^{\perp})$ are unitarily equivalent if $M$ and $N$ are subspaces of $\comp$ in generic position. We compute numerical ranges of sum and product of two orthogonal projections by using Jordan planes.
\end{abstract}

\maketitle

\section{Principal Angles}
In 1875, C. Jordan \cite{Jordan} introduced the principal angles between two subspaces. Since then, many attempts have been made to develop and derive recursive definitions of principal angles. The principal angles along with principal vectors provide optimal approaches in describing many properties involving two subspaces. A great deal of discussion on principal angles was made by A. Galantai and C.J.Hegedus in \cite{Galantai3}. \\

\begin{definition}
    Let $M$ and $N$ be two subspaces of $\comp$ with $\dim M=p$ and $\dim N =q$. Assume $p \geq q$. The \emph{principal angles $\theta_1, \theta_2, \cdots \, \theta_q$ between $M$ and $N$} can be constructed recursively as follows.
       \begin{align*}
      \cos \theta_i =& \sup \{|\langle u, v \rangle| \, : \, u \in M \cap M_{i-1}^{\perp} , v \in N \cap N_{i-1}^{\perp} \, , \|u\|=\|v\|=1\} \\
                    =& \langle u_i, v_i \rangle \,\, \textrm{for some unit vectors} \, u_i \in M \cap M_{i-1}^{\perp}\, \textrm{and} \, v_i \in N \cap N_{i-1}^{\perp},\, \\
                     & \textrm{for} \,\ i=1,\cdots,q
    \end{align*}
    where $M_{i-1}$ and $N_{i-1}$ are subspaces spanned by $\{u_1,\cdots,u_{i-1}\}$ and $\{v_1,\cdots,v_{i-1}\}$, respectively, and $M_0=N_0=\{0\}$.
\end{definition}

  The vectors $u_1,u_2, \cdots , u_q$ and $v_1,v_2, \cdots , v_q$ are called \emph{the principal vectors}. The principal angles are uniquely determined for a pair of subspaces, but the principal vectors are not. Some properties associated with the principal angles follow from the construction.
  \begin{enumerate}
    \item $0 \leq \theta_1 \leq \theta_2 \leq  \cdots \leq \theta_q \leq \frac{\pi}{2} $.
    \item $\theta_1 =0 $ if and only if $M \cap N \neq \{0\}$. More specifically, if $\dim (M\cap N) =r$, then $\theta_1=\cdots=\theta_r=0$.
    \item $\theta_q=\frac{\pi}{2}$ if and only if $(M \cap N^{\perp}) \oplus (M^{\perp} \cap N) \neq \{0\}$. 
    %Especially, if $\dim(M^{\perp} \cap N)=s$, then $\theta_{q-s+1}=\cdots=\theta_{q}=\frac{\pi}{2}$.
   % \item If $M$ and $N$ are in generic position, then $p=q$ and $\theta_1 >0$ and $\theta_q <\frac{\pi}{2} $.
 
  \end{enumerate}

The following biorthogonality relation of principal vectors is worthwhile to note. The relation will play a vital role in computing numerical range of product of orthogonal projections. A proof of the lemma can be found in \cite{Galantai3}. 

\begin{lemma}\label{Biortho}
  Assume that $u_1,u_2, \cdots , u_q$ and $v_1,v_2, \cdots , v_q$ are principal vectors obtained in the recursive definition of principal angles, $\theta_1, \theta_2, \cdots \, \theta_q$, of two subspaces $M$ and $N$. Then
  $$\langle u_i, v_j \rangle = \delta_{ij}\cos \theta_i \,\, \textrm{for} \,\, i,j =1,2,\cdots ,q.$$
\end{lemma}

Among all the principal angles, two of them assume more importance. $\theta_{r+1}$ is referred to as the \emph{Friedrichs angle} $\alpha(M,N)$ between $M$ and $N$, and $\theta_1$ is referred to as the \emph{Dixmier angle} $\alpha_0(M,N)$ between $M$ and $N$. Note that for each $k$ with $\theta_k \in (0, \frac{\pi}{2})$, there exist two nonparallel and nonperpendicular vectors $u_k \in M$ and $v_k \in N$.

The collection of principal vectors provides a real nice way of representing one subspace in terms of the other. Readers may refer to \cite{Galantai3} for the proof of the following lemma. 

\begin{lemma}\label{TwoProjs}
Assume that $M$ and $N$ are two subspaces of $\comp$ and $P_M$ and $P_N$ are orthogonal projections onto $M$ and $N$, respectively. If $u_1,u_2, \cdots , u_q$ and $v_1,v_2, \cdots , v_q$ are principal vectors corresponding to principal angles   
, $\theta_1, \theta_2, \cdots \, \theta_q$, of two subspaces $M$ and $N$, then
  $$P_Mv_j=\cos \theta_j u_j \quad \textrm{and} \quad P_Nu_j=\cos \theta_j v_j \quad \textrm{for} \quad j=1,2,\cdots, q.$$ 
\end{lemma}

The results in Lemma \ref{Biortho} and Lemma \ref{TwoProjs} allow us to break down $\comp$ and to decompose it in an optimal fashion in terms of two subspaces. The next definition will be used as the building blocks to express many of the results in this paper.    

\begin{definition}
  If $0<\theta_k<\frac{\pi}{2}$, the two dimensional subspace spanned by $u_k$ and $v_k$ is called the \emph{$k$-th Jordan plane associated with $M$ and $N$}. $J_k$ denotes the $k$-th Jordan plane.
\end{definition}

An orthogonal decomposition of $\comp$ in terms of two subspaces $M$ and $N$ can be given as follows:

 \begin{align*}
   \comp=(M \cap N) \oplus (M \cap N^{\perp}) \oplus (M^{\perp} \cap N) \oplus (M^{\perp} \cap N^{\perp}) \oplus R ,
 \end{align*}
 where $R$ denotes the orthogonal complement of $(M \cap N) \oplus (M \cap N^{\perp}) \oplus (M^{\perp} \cap N) \oplus (M^{\perp} \cap N^{\perp})$.

\begin{definition}
 Two subspaces $M$ and $N$ of $\comp$ are said to be in \emph{generic position} if $$(M \cap N) \oplus (M \cap N^{\perp}) \oplus (M^{\perp} \cap N) \oplus (M^{\perp} \cap N^{\perp})=0.$$
 \end{definition}

 $R_M$ and $R_N$ will denote $R\cap M$ and $R \cap N$, respectively.   

Observe that Jordan planes are two dimensional subspaces and mutually orthogonal by Lemma \ref{Biortho}, and if two subspaces $M$ and $N$ are in generic position with $\dim M = \dim N=p$, then there are $p \,$ Jordan planes.

 % Now let $\eta_1, \eta_2, \cdots, \eta_d$ be the principal angles between $M^{\perp}$ and $N^{\perp}$, where $d=\min \{\dim M^{\perp}, \dim N^{\perp}\}$.
 % We want to derive the relationships between $\{\theta_1, \theta_2, \cdots \, \theta_q\}$ and $\{\eta_1, \eta_2, \cdots, \eta_d\}$. 
 Over the Jordan plane $J_k$ formed by the two principal vectors $u_k$ and $v_k$, we can establish \emph{trigonometric identities} by taking the steps described below.

  First, let $s_k$ be a unit vector in the $k$-th Jordan plane perpendicular to $u_k$ and $\langle v_k, s_k \rangle > 0$. Note that $\langle v_k,s_k \rangle <1$. Second, we choose a unit vector $t_k \in J_k$  satisfying $\langle v_k, t_k \rangle =0$ and $\langle s_k, t_k \rangle >0$. Then the four vectors $u_k,v_k,s_k,t_k \in J_k$ hold the following properties.

\begin{lemma}\label{four vectors}
 Let $J_k$ be the $k$-th Jordan plane associated with a principal angle $0<\theta_k<\frac{\pi}{2}$ and let $u_k$ and $v_k$ be principal vectors corresponding to $\theta_k$, i.e. $\cos \theta_k =\langle u_k, v_k \rangle$. 
 Then the four unit vectors $u_k,v_k,s_k,t_k \in J_k$ constructed in the preceding argument satisfy the following properties.
% If $\{u_k\}_{k=1}^{p}$, $\{v_k\}_{k=1}^{p}$, $\{s_k\}_{k=1}^{p}$, and $\{t_k\}_{k=1}^{p}$ are four sets of vectors constructed as above for two subspaces $M$ and $N$ in generic position, then the following are true.
  \begin{enumerate}
    \item $\langle u_k, v_k \rangle^2+\langle v_k, s_k \rangle^2=1$ and $\langle v_k, s_k \rangle^2+\langle s_k, t_k \rangle^2=1$. 
    \item $\langle u_k, v_k \rangle = \langle s_k, t_k \rangle=\cos \theta_k$.
    \item $\langle v_k, s_k \rangle = -  \langle u_k, t_k \rangle=\sin \theta_k$. 

  \end{enumerate}
\end{lemma}

\begin{proof}
  \begin{enumerate}
   \item Note that $\{u_k, s_k \}$ forms an orthonormal basis for $J_k$, so we can write $v_k=\alpha u_k + \beta s_k$, where $\alpha = \langle v_k, u_k \rangle$ and $\beta =\langle v_k, s_k \rangle$. Since $v_k$ is a unit vector, $\langle u_k, v_k \rangle^2+\langle v_k, s_k \rangle^2=|\alpha|^2+|\beta|^2=|v_k|^2=1$. The identity $\langle v_k, s_k \rangle^2+\langle s_k, t_k \rangle^2=1$ can be proven similarly. 
   \item It is easy to see from $(1)$ that $\langle u_k, v_k \rangle^2=\langle s_k, t_k \rangle^2$. Since both $\langle u_k, v_k \rangle$ and $\langle s_k, t_k \rangle$ are positive, $\langle u_k, v_k \rangle=\langle s_k, t_k \rangle$. 
   \item Since $v_k \perp t_k$, 
    \begin{align*} 
    \langle v_k, t_k \rangle &= \langle \langle v_k, s_k \rangle s_k + \langle v_k, u_k \rangle u_k, t_k \rangle \\
    						&= \langle v_k, s_k \rangle \langle s_k, t_k \rangle + \langle v_k, u_k \rangle  \langle u_k, t_k \rangle \\
           					&= 0.                  				
    \end{align*}
    Since $\langle v_k, u_k \rangle = \langle u_k, v_k \rangle = \langle s_k, t_k \rangle \neq 0$, we have $\langle v_k, s_k \rangle +  \langle u_k, t_k \rangle =0$.
  \end{enumerate}
\end{proof}

Observe that if two subspaces $M$ and $N$ of $\mathbb{C}^n$ are in generic position, then $\dim M = \dim M^{\perp} =\dim N = \dim N^{\perp}=p$, where $2p=n$. Clearly, the collections of unit vectors $\{u_k\}_{k=1}^p$ and $\{v_k\}_{k=1}^p$, where $u_k$ and $v_k$ are principal vectors corresponding to the principal angle $\theta_k$ between $M$ and $N$, form orthonormal bases for $M$ and $N$, respectively. Since the pair $(u_k,v_k)$ spans the $k$-th Jordan plane, $J_k$, for each $k=1,\cdots,p$, there are $p$ Jordan planes that are mutually orthogonal and $\oplus_{k=1}^p J_k =\mathbb{C}^n $.   

\begin{lemma} \label{four bases}
 Assume that $M$ and $N$ are subspaces of $\mathbb{C}^n$ in generic position. Let $\dim M = \dim M^{\perp} =\dim N = \dim N^{\perp}=p$, where $2p=n$. 
 
 Consider the four collections of unit vectors $\{u_k\}_{k=1}^{p}$, $\{v_k\}_{k=1}^{p}$, $\{s_k\}_{k=1}^{p}$, and $\{t_k\}_{k=1}^{p}$, where $u_k,v_k,s_k$, and $t_k$ are four vectors constructed in Lemma \ref{four vectors} associated with $J_k$. Then the following are true.
 
% If $\{u_k\}_{k=1}^{p}$, $\{v_k\}_{k=1}^{p}$, $\{s_k\}_{k=1}^{p}$, and $\{t_k\}_{k=1}^{p}$ are four sets of vectors constructed as above for two subspaces $M$ and $N$ in generic position, then the following are true.
  \begin{enumerate}

  \item $\{u_k\}_{k=1}^{p}$, $\{v_k\}_{k=1}^{p}$, $\{s_k\}_{k=1}^{p}$, and $\{t_k\}_{k=1}^{p}$ are orthonormal bases for $M$, $N$, $M^{\perp}$, and $N^{\perp}$, respectively.
  \item $\langle u_i, v_j \rangle = \langle s_i, t_j \rangle = \delta_{ij}\cos \theta_i$ and $\langle s_i, v_j \rangle = -\langle u_i, t_j \rangle = \delta_{ij}\sin \theta_i$.
  \item $u_k=(\cot \theta_k) s_k - (\csc \theta_k) t_k$.
  \end{enumerate}
\end{lemma}

\begin{proof}

   \begin{enumerate}
      \item Note that $s_k \perp u_k$ by the definition of $s_k$ and $s_k \perp u_i$ for $i \neq k$ by the mutual orthogonality of Jordan planes, so $s_k \in M^{\perp}$ for each $k=1,\cdots, p$. Therefore, $\{s_k\}_{k=1}^{p}$ is an orthonormal basis for $M^{\perp}$. Similarly, $\{t_k\}_{k=1}^{p}$ forms an orthonormal basis for $N^{\perp}$. 
     
     \item The results follow from Lemma \ref{Biortho} and Lemma \ref{four vectors}.
     
     \item  Note that $\{s_k,t_k\}$ is linearly independent, so we can write $u_k=\alpha s_k + \beta t_k $ for some $\alpha, \beta \in \mathbb{C}$. To determine $\alpha$ and $\beta$, we compute
     
         \begin{align*}
    0=\langle u_k, s_k \rangle =& \langle \alpha s_k + \beta t_k, s_k \rangle  \\
    =& \alpha \langle s_k, s_k  \rangle + \beta \langle t_k, s_k  \rangle \\
    =& \alpha + \beta \cos \theta_k.
  \end{align*}
  It follows from \ref{four bases} (3) that
      \begin{align*}
    -\sin \theta_k =\langle u_k, t_k \rangle =& \langle \alpha s_k + \beta t_k, t_k \rangle  \\
    =& \alpha \langle s_k, t_k  \rangle + \beta \langle t_k, t_k  \rangle \\
    =& \alpha \cos \theta_k + \beta.
  \end{align*}
  Solving the system of equations in $\alpha$ and $\beta$, we obtain $\alpha =\cot \theta_k$ and $\beta=-\csc \theta_k$.
     
     \end{enumerate}

\end{proof}

Two pairs of subspaces $(M_1,N_1)$ and $(M_2,N_2)$ are said to be \emph{unitarily equivalent} if there exists a unitary operator $U:\comp \to \comp$ such that $UM_1=M_2$ and $UN_1=N_2$. Jordan proved the following theorem in \cite{Jordan}.
\begin{theorem}\label{jordan}
       Two pairs of subspaces $(M_1,N_1)$ and $(M_2,N_2)$ are unitarily equivalent if and only if the following are true.
       \begin{enumerate}
         \item $ \dim(M_1 \cap N_1)=\dim(M_2 \cap N_2)$, $\dim(M_1 \cap \npp_1)=\dim(M_2 \cap \npp_2)$, $\dim(\mpp_1 \cap N_1)=\dim(\mpp_2 \cap N_2)$, $\dim(\mpp_1 \cap \npp_1)=\dim(\mpp_2 \cap \npp_2)$, $\dim(R_1)=\dim(R_2)$.
         \item The principal angles between $M_1$ and $N_1$ are equal to the principal angles between $M_2$ and $N_2$.
         
         \end{enumerate}

\end{theorem}\label{Jordan}

 In the proof of the following theorem, we use the same collections of unit vectors introduced earlier. The symbol $x \otimes y^*$ denotes the rank-one operator for $x,y \in \comp$ defined by $(x \otimes y^*)(f)=\langle f, y \rangle x$ for $f \in \comp$.   
\begin{theorem}\label{uteq}
  If $M$ and $N$ are subspaces of $\mathbb{C}^n$ in generic position, then $(M,N)$ and $(M^{\perp},N^{\perp})$ are unitarily equivalent.
\end{theorem}
\begin{proof}
  Let $\dim M = \dim M^{\perp} =\dim N = \dim N^{\perp}=p$ where $p=\frac{n}{2}$. 
  
  Define a mapping $U:\comp \to \comp$ by
  $$U=\sum_{k=1}^{p} \csc \theta_k ( t_k \otimes s_k^* - s_k \otimes t_k^*).$$
  It is easy to see that $U^*=-U$. 
  We now compute images of the unit vectors under $U$.  
  \begin{align*}
    Uu_j=&\left(\sum_{k=1}^{p} \csc \theta_k ( t_k \otimes s_k^* - s_k \otimes t_k^*)\right)(u_j) \\
       =& \left(\sum_{k=1}^{p} \csc \theta_k ( \langle u_j, s_k \rangle t_k - \langle u_j, t_k \rangle s_k)\right) \\
       =& \csc \theta_j (-\langle u_j, t_j \rangle s_j) \\
       =& \csc \theta_j \sin \theta_j s_j \\
       =& s_j
  \end{align*}
  This shows that $UM=M^{\perp}$. Similarly, we can compute $Uv_j=t_j$, which implies $UN=N^{\perp}.$ 
  
  Now,
  \begin{align*}
   Us_j=&\left(\sum_{k=1}^{p} \csc \theta_k ( t_k \otimes s_k^* - s_k \otimes t_k^*)\right)(s_j) \\
       =& \left(\sum_{k=1}^{p} \csc \theta_k ( \langle s_j, s_k \rangle t_k - \langle s_j, t_k \rangle s_k)\right) \\
       =& \csc \theta_j (t_j-\langle s_j, t_j \rangle s_j) \\
       =& \csc \theta_j t_j - \csc \theta_j \cos \theta_j s_j \\
       =& \csc \theta_j t_j - \cot \theta_j s_j \\
       =& -u_j \,\,\,\, \textrm{by Lemma\ref{four bases}} (3).
  \end{align*}
%  This implies that $UN=N^{\perp}.$
  
  Combining the two results, we obtain $U^*Uu_j=U^*s_j=-Us_j=u_j$ and $U^*Us_j=U^*(-u_j)=Uu_j=s_j$. 
  Since $\{u_j,s_j\}_{j=1}^p$ is an orthonormal basis for $\comp$, $U$ is a unitary operator. Therefore, $(M,N)$ and $(M^{\perp},N^{\perp})$ are unitarily equivalent.
\end{proof}

 \begin{cor}
 Assume that two subspaces $M$ and $N$ of $\comp$ are in generic position with $\dim M = \dim M^{\perp} =\dim N = \dim N^{\perp}=p$, where $p=\frac{n}{2}$. Let $\{\theta_k\}_{k=1}^p$ be principal angles between $M$ and $N$ and let $\{u_k\}_{k=1}^{p} \subset M$ and $\{v_k\}_{k=1}^{p} \subset N$ be principal vectors corresponding to $\{\theta_k\}_{k=1}^p$, i.e. $\langle u_k, v_k \rangle =\cos \theta_k$. 
 
 If we let $\{s_k\}_{k=1}^{p}$ and $\{t_k\}_{k=1}^{p}$ be the orthonormal bases for $M^{\perp}$ and $N^{\perp}$ as constructed in Lemma \ref{four vectors}, then $\{\theta_k\}_{k=1}^p$ are principal angles $M^{\perp}$ and $N^{\perp}$, and $\{s_k\}_{k=1}^{p}$ and $\{t_k\}_{k=1}^{p}$ are corresponding principal vectors.

 \end{cor}
 \begin{proof}
  The results follow from Theorem \ref{jordan} and Theorem \ref{uteq}. 
 \end{proof}

 $M$ and $N$ are said to be \emph{in generalized generic position} if $M \cap N = \mpp \cap \npp =\{0\}$ and $\dim (M \cap \npp) = \dim (\mpp \cap N)$.
\begin{cor}
  Let $M$ and $N$ be in generalized generic position with $\dim M =\dim N =p$.
  If $\{\theta_i\}_{i=1}^{p}$ and $\{\eta_i\}_{i=1}^{p}$ are principal angles between $M$ and $N$ and between $M^{\perp}$ and $N^{\perp}$, respectively, then $\theta_i=\eta_i$ for $i=1,2,\cdots, p$.
\end{cor}
\begin{proof}
  If we let $\dim (R_M \oplus R_{\mpp})=\dim(R_N \oplus R_{\npp})=r$, then $\dim (M \cap \npp) = \dim (\mpp \cap N)= p-r $.
  It follows from Theorem \ref{uteq} that $\theta_i=\eta_i$ for $i \in \{1,2,\cdots, r\}$. For $i \in \{r+1,\cdots,p\}$, $\theta_i=\eta_i=\frac{\pi}{2}.$
\end{proof}

 Let $\dim M = p$, $\dim N = q$, $\dim (M \cap N) = a$, $\dim (\mpp \cap \npp) = b$, $\dim (M \cap \npp)=c$, $\dim (\mpp \cap N)=d$, and $\dim (R_M \oplus R_{\mpp}) = \dim (R_N \oplus R_{\npp})=2r$. Then we have $p=a+c+r$ and $q=a+d+r$. Also, we can see $\dim \mpp = n-p=b+d+r$ and $\dim \npp = n-q=b+c+r$. If we assume $p \geq q$, then $\dim \mpp = n-p \leq \dim \npp = n-q$. Let $\{\theta_1, \theta_2, \cdots \, \theta_q\}$ be the principal angles between $M$ and $N$. And let $\{\eta_1, \eta_2, \cdots, \eta_{n-p}\}$ be the principal angles between $\mpp$ and $\npp$. We have the following general relationship between the two sets of principal angles.
 \begin{enumerate}
   \item $\theta_1=\cdots=\theta_a=\eta_1=\cdots=\eta_b=0$.
   \item $\theta_{a+1}=\eta_{b+1},\theta_{a+2}=\eta_{b+2},\cdots,\theta_{a+r}=\eta_{b+r}$
   \item $\theta_{a+r+1}=\cdots=\theta_{a+r+c}=\eta_{b+r+1}=\cdots=\eta_{b+r+d}=\frac{\pi}{2}$
 \end{enumerate}

\section{Numerical Range of $P_M +P_N$}
In this section we discuss the numerical range of the operators involving two orthogonal projections. Principal angles and Jordan planes can be useful in determining the numerical ranges of those operators.

If $A \in M(\comp)$, the numerical range of $A$, denoted by $W(A)$, is defined by
$$W(A)=\{\langle Ax, x \rangle : x \in \comp, \|x\|=1\}.$$ The numerical radius of $A$, $w(A)$, is defined to be $w(A)=\sup\{|\lambda| : \lambda \in W(A)\}$. The next proposition contains some standard properties of $W(A)$ and $w(A)$. A detailed description, examples, and proofs of the following standard properties of numerical range can be found in \cite{Shapiro}. 
\begin{proposition}\label{NumRan}
  If $A,B \in M(\comp)$, then the following are true.
  \begin{enumerate}
  	\item (Hausdorff-Toeplitz) $W(A)$ is convex.
    \item $W(A^*)=\{\overline{\lambda} : \lambda \in W(A) \}.$
    \item $w(A) \leq \|A\|$.
    \item $W(A)$ contains all eigenvalues of $A$.
    \item\label{ute} If $A$ and $B$ are unitarily equivalent, then $W(A)=W(B)$.
    \item If $A$ is Hermitian i.e. $A=A^*$, then $W(A) \subset \mathbb{R}$ and $w(A)=\|A\|$.
    \item\label{abi} If $\alpha,\beta \in \mathbb{C}$, then $W(\alpha A + \beta I)=\alpha W(A) + \beta$.
    \item $W(A)$ is compact.
    \item If $A=B \oplus C$, then $W(A)$ is the convex hull of $W(B) \cup W(C)$.
    
  \end{enumerate}
\end{proposition}

\begin{lemma}\label{decomp}
  If $A \in M(\comp)$, then $W(A) \subset \{a+ib: a \in Re(A) \,\, \textrm{and}\,\, b \in Im(A)\}$, where $Re(A)=\frac{A+A^*}{2}$ and $Im(A)=\frac{A-A^*}{2i}$.
\end{lemma}
\begin{proof}
  Straightforward.
\end{proof}

 First, we compute the numerical range of the sum of two orthogonal projections. Let $M$ and $N$ be two subspaces of $\comp$, and $P_M$ and $P_N$ be orthogonal projections onto $M$ and $N$, respectively. By Proposition \ref{NumRan}, $W(P_M+P_N)$ is a closed bounded interval in $\mathbb{R}$.  We first note that $\|P_Mx\|^2+\|P_Nx\|^2 =\langle (P_M+P_N) x, x \rangle $. An inequality showing the bounds of $\|P_Mx\|^2+\|P_Nx\|^2$ in terms of Dixmier angles can be obtained.

The following two propositions will be useful in computing $W(P_M+P_N)$.
  \begin{proposition}\label{NSP}
  If $P$ and $Q$ are orthogonal projections on $\comp$,
    $$\|P+Q\|=1+\|PQ\|.$$
  \end{proposition}

   \begin{proposition}\label{NPP}
     If $P_M$ and $P_N$ are orthogonal projections onto $M$ and $N$, respectively, then $$\|P_MP_N\|=\cos (\alpha_0(M,N))=\cos \theta_1. $$
   \end{proposition}
   Proposition \ref{NSP} can be found in \cite{Duncan} and \cite{Vidav}, and Proposition \ref{NPP} in \cite{Deutsch}. By combining the two propositions, we obtain $$\|P_M+P_N\|=1+\cos \theta_1.$$

 \begin{theorem}\label{BOSS}
   Let $M$ and $N$ be two subspaces of $\comp$. %with $\dim M =p \geq q = \dim N$. 
   Let $\theta_1= \alpha_0(M,N)$ and $\eta_1=\alpha_0(\mpp,\npp)$ be the Dixmier angles between $M$ and $N$ and between $M^{\perp}$ and $N^{\perp}$, respectively. For $x \in \comp$,
   $$2\|x\|^2\sin^2\frac{\eta_1}{2} \leq \|P_Mx\|^2+\|P_Nx\|^2 \leq 2\|x\|^2\cos^2\frac{\theta_1}{2}.$$
 \end{theorem}
 \begin{proof}
 The right side inequality comes from the following.
   \begin{align*}
     \|P_Mx\|^2+\|P_Nx\|^2 =& \langle (P_M+P_N)x,x \rangle  \\
                           \leq & \|(P_M+P_N)x\|\|x\| \\
                           \leq & \|P_M+P_N\|\|x\|^2 \\
                           =& (1+\|P_MP_N\|)\|x\|^2  \,\,\,\, \textrm{by Proposition \ref{NSP}} \\
                           =& (1+\cos \theta_1)\|x\|^2  \,\,\,\, \textrm{by Proposition \ref{NPP}} \\
                           =& 2\cos^2\frac{\theta_1}{2} \|x\|^2.
   \end{align*}
   Next, we show that the equality can be attained for some $x$. Letting $x=u_1+v_1$,
      \begin{align*}
     \langle (P_M+P_N)x,x \rangle =& \langle (P_M+P_N)(u_1+v_1), u_1+v_1 \rangle \\
     							 =& \langle u_1+P_Mv_1+P_Nu_1+v_1, u_1+v_1 \rangle \\
                                 =& \langle u_1+\cos \theta_1 u_1+v_1+\cos \theta_1 v_1, u_1+v_1 \rangle  \,\,\,\, \textrm{by Lemma \ref{TwoProjs}} \\
                                 =& 2(1+\cos \theta_1)+2(1+\cos\theta_1)\cos\theta_1 \\
                                 =& 2(1+\cos \theta_1)(1+\cos \theta_1) \\
                                 =& (1+\cos\theta_1)\|u_1+v_1\|^2=2\cos^2 \frac{\theta_1}{2}\|x\|^2.
   \end{align*}
   
    On the other hand,
    \begin{align*}
     2\|x\|^2- (\|P_Mx\|^2+\|P_Nx\|^2) =& (\|x\|^2-\|P_Mx\|^2)+(\|x\|^2-\|P_Nx\|^2)  \\
     =& \|(I-P_M)x\|^2+\|(I-P_N)x\|^2 \\
     =& \langle (I-P_M+I-P_N)x,x \rangle \\
     \leq & \|I-P_M+I-P_N\|\|x\|^2 \\
     =& (1+\|(I-P_M)(I-P_N)\|)\|x\|^2 \\
     =& (1+\cos \eta_1)\|x\|^2 = 2\cos^2\frac{\eta_1}{2} \|x\|^2.
    \end{align*}
    Therefore, we have $$\|P_Mx\|^2+\|P_Nx\|^2 \geq 2(1-\cos^2\frac{\eta_1}{2})\|x\|^2=2\|x\|^2\sin^2\frac{\eta_1}{2}.$$
    It can be shown that the equality holds when $x=s_1+t_1$. 
 \end{proof}

\begin{cor}
If $M$ and $N$ are subspaces of $\comp$, then
  $$W(P_M+P_N)=\left[2\sin^2\frac{\eta_1}{2}, 2\cos^2\frac{\theta_1}{2}\right]\,\,\, \textrm{and} \,\,\, w(P_M+P_N)=2\cos^2\frac{\theta_1}{2}.$$
\end{cor}

\section{Numerical Range of Product of Projections}
In this section we describe the eigenvalues and eigenvectors of several operators involving $P_M$ and $P_N$ in terms of principal angles and principal vectors. Assume that $M$ and $N$ are in generic position with $\dim M =\dim N =p$. For the sake of convenience we write $\lambda_k= \cos \theta_k = \langle u_k, v_k \rangle = \langle s_k, t_k \rangle$ and $\mu_k= \sin \theta_k = \langle s_k, v_k \rangle = -\langle u_k, t_k \rangle$ for $k=1,\cdots,p$. Note that $\lambda_k,\mu_k \in (0,1)$ and $\lambda_k^2+\mu_k^2=1$. We will write $P$ and $Q$ in place of $P_M$ and $P_N$, respectively. The notation $EV(A)=\{(\lambda,f):\lambda \in \sigma(A) \, \textrm{and} \, Af=\lambda f\}$ will be used to denote the collection of the pairs of eigenvalues and eigenvectors of $A$.

\begin{lemma}\label{EV}
  Let $M$ and $N$ be two subspaces of $\comp$ in generic position with $\dim M =\dim N=p$. If $P=P_M$ and $Q=P_N$, then the pairs of eigenvalues and eigenvectors for the operators involving $P$ and $Q$ can be given as follows:
  \begin{enumerate}
    \item \label{P+Q} $\displaystyle EV(P+Q)=\{(1+\lambda_k,u_k+v_k), (1-\lambda_k,u_k-v_k)\}_{k=1}^p$,
    \item \label{P-Q} $\displaystyle EV(P-Q)=\{(\mu_k,u_k-\frac{1-\mu_k}{\lambda_k}v_k), (-\mu_k,u_k-\frac{1+\mu_k}{\lambda_k}v_k)\}_{k=1}^p$,
    \item \label{PQ} $EV(PQ)=\{(\lambda_k^2,u_k), (0,t_k)\}_{k=1}^p$ and $EV(QP)=\{(\lambda_k^2,v_k), (0,s_k)\}_{k=1}^p$,
   % \item \label{PQP} $EV(PQP)=\{(\lambda_k^2,u_k), (0,s_k)\}_{k=1}^p$ and $EV(QPQ)=\{(\lambda_k^2,v_k), (0,t_k)\}_{k=1}^p$
    \item \label{PQ+QP} $EV(PQ+QP)=\{(\lambda_k^2+\lambda_k,u_k+v_k), (\lambda_k^2-\lambda_k,u_k-v_k)\}_{k=1}^p$,
    \item \label{PQ-QP} $EV(PQ-QP)=\{(i\lambda_k\mu_k,u_k-e^{-i\theta_k }v_k), (-i\lambda_k\mu_k,u_k-e^{i\theta_k }v_k)\}_{k=1}^p$.
  \end{enumerate}
  \end{lemma}
  \begin{proof}
  \begin{enumerate}
   \item Recall that $Pv_k=  \langle v_k, u_k \rangle u_k= \lambda_k u_k$ and $Qu_k =\lambda_k v_k$ from Lemma \ref{TwoProjs}.
   The result follows from  \begin{align*}
   (P+Q)(u_k+v_k)&= Pu_k+Pv_k+Qu_k+Qv_k \\
                 &=(1+\lambda_k)(u_k+v_k)   
   \end{align*}
   
   and \begin{align*}
   (P+Q)(u_k-v_k)&=Pu_k-Pv_k+Qu_k-Qv_k \\
                 &=u_k-\lambda_k u_k +\lambda_k v_k - v_k \\
                 &=(1-\lambda_k)(u_k-v_k). 
   \end{align*}
   
    \item To compute the eigenvalues and eigenvectors of $P-Q$,
     \begin{align*}
     (P-Q)\left(u_k-\frac{1-\mu_k}{\lambda_k}v_k\right) = & Pu_k-P\left(\frac{1-\mu_k}{\lambda_k}v_k\right)-Qu_k+Q\left(\frac{1-\mu_k}{\lambda_k}v_k\right)  \\
                           =& u_k-\left(1-\mu_k\right)u_k-\lambda_k v_k + \frac{1-\mu_k}{\lambda_k}v_k \\
                           =& \mu_k u_k +\left(-\lambda_k + \frac{1-\mu_k}{\lambda_k}\right) v_k \\
                           =& \mu_k u_k +\left(\frac{-\lambda_k^2+1-\mu_k}{\lambda_k}\right) v_k \\
                           =& \mu_k u_k +\left(\frac{\mu_k^2-\mu_k}{\lambda_k}\right) v_k \\
                           =& \mu_k \left(u_k +\left(\frac{\mu_k-1}{\lambda_k}\right) v_k\right) = \mu_k \left(u_k -\left(\frac{1-\mu_k}{\lambda_k}\right) v_k\right). \\
        \end{align*}
   Similarly,  $$(P-Q)\left(u_k-\frac{1+\mu_k}{\lambda_k}v_k\right) = -\mu_k \left(u_k-\frac{1+\mu_k}{\lambda_k}v_k\right).$$
  \item  Since $PQu_k=P(\lambda_k v_k)=\lambda_k^2 u_k$ and $PQt_k=0$,  $EV(PQ)=\{(\lambda_k^2,u_k), (0,t_k)\}_{k=1}^p$. Similarly, $EV(QP)=\{(\lambda_k^2,v_k), (0,s_k)\}_{k=1}^p$. 
  
  \item   Note that 
     \begin{align*}
     (PQ+QP)(u_k+v_k)=&\lambda_k^2 u_k + \lambda_k u_k + \lambda_k^2 v_k + \lambda_k v_k\\
                    =& (\lambda_k^2+\lambda_k)(u_k+v_k),
     \end{align*}
     and 
     \begin{align*}
     (PQ+QP)(u_k-v_k)=&\lambda_k^2 u_k - \lambda_k u_k - \lambda_k^2 v_k + \lambda_k v_k \\ 
                     =& (\lambda_k^2-\lambda_k)(u_k-v_k).
     \end{align*} 
    Hence, $EV(PQ+QP)=\{(\lambda_k^2+\lambda_k,u_k+v_k), (\lambda_k^2-\lambda_k,u_k-v_k)\}_{k=1}^p$.
    
    \item Compute
     \begin{align*}
       (PQ-QP)(u_k-e^{-i\theta_k}v_k) =&PQ(u_k-e^{-i\theta_k}v_k)-QP(u_k-e^{-i\theta_k}v_k) \\
       						=& P(\cos \theta_k v_k -e^{-i\theta_k}v_k)-Q(u_k-\cos \theta_k e^{-i\theta}u_k) \\
                            =& i \sin \theta_k \cos \theta_k u_k-(\cos \theta_k \sin^2 \theta_k +i \cos^2 \theta_k \sin \theta_k)v_k \\
                            =& i \sin \theta_k \cos \theta_k u_k-i \sin \theta_k \cos \theta_k e^{-i\theta_k}v_k \\
                            =& i \sin \theta_k \cos \theta_k (u_k-e^{i\theta_k}v_k) =i \mu_k \lambda_k (u_k-e^{i\theta_k}v_k).
       \end{align*}
       Similarly, we can show that $(PQ-QP)(u_k-e^{i\theta_k}v_k)=-i\mu_k \lambda_k (u_k-e^{i\theta_k}v_k).$
    \end{enumerate}
  \end{proof}

We turn our attention to determining the numerical range of $PQ$. Since 
$$PQ=\frac{PQ+QP}{2}+i\frac{PQ-QP}{2i}$$ where $\frac{PQ+QP}{2}$ and $\frac{PQ-QP}{2i}$ are Hermitian, by Lemma \ref{decomp}, we have $W(PQ) \subset \{x+iy : x \in W(\frac{PQ+QP}{2})+i W(\frac{PQ-QP}{2i})\}$. 

Next we prove that the numerical range of $PQ$ on a Jordan plane is an elliptical disc. The result is quite obvious when we consider that the range of $PQ|_{J_k}$ is a subspace of $J_k$. Hence, $PQ|_{J_k}$ can be represented in a $2 \times 2$ matrix form. The Elliptic Range Theorem states that the numerical range of such operator is an (possibly degenerate) elliptic disk. However, we want to elaborate all the steps to be able to see more details in the process. 

%$P=\displaystyle \sum_{j=1}^p u_j \otimes u_j^*$ and $Q=\displaystyle \sum_{j=1}^p v_j \otimes v_j^*$, $PQ|_{J_k}=\langle u_k, v_k \rangle u_k \otimes v_k^*$,

\begin{lemma}\label{NRcd}
  If $A=\begin{bmatrix}
    0 & a \\
    b & 0
  \end{bmatrix}$ where $0<b<a$, then $W(A)$ is an elliptic disc centered at the origin whose foci are $\pm\sqrt{ab}$, whose major axis is of length $a+b$, and the minor axis is of length $a-b$.
\end{lemma}
\begin{proof}
  Let $\xi=\colvec{(\cos t) e^{i \phi_1}}{(\sin t) e^{i \phi_2}}$ where $t \in [0, 2\pi]$ and $\phi_1, \phi_2 \in [0, 2\pi]$ be a parametrization for a unit vector in a two dimensional space complex Hilbert space. Note that

       \begin{align*}
     \langle A \xi,\xi \rangle = & \left\langle A\colvec{(\cos t) e^{i \phi_1}}{(\sin t) e^{i \phi_2}}, \colvec{(\cos t) e^{i \phi_1}}{(\sin t) e^{i \phi_2}} \right\rangle \\
                               = & a (\sin t \cos t) e^{i\phi_2} e^{-i \phi_1}+b (\cos t \sin t) e^{i\phi_1} e^{-i\phi_2} \\
                               = & \sin t \cos t \left(a e^{i(\phi_2-\phi_1)}+b e^{-i(\phi_2-\phi_1)}\right) \\
                               = & \sin t \cos t \left((a+b)\cos (\phi_2-\phi_1) +i(a-b) \sin (\phi_2-\phi_1)\right).
        \end{align*}
     Therefore, $$W(A)=\left\{ \frac{\sin (2t)}{2}\left( (a+b) \cos \phi + i (a-b) \sin \phi)\right) : t \in [0, 2\pi], \phi \in [0, 2\pi]\right\}.$$
     If $x=(a+b) \frac{\sin (2t)}{2} \cos \phi$ and $y=(a-b) \frac{\sin (2t)}{2} \sin \phi$, then we obtain the equation $$ \frac{x^2}{\left(\frac{a+b}{2}\right)}+ \frac{y^2}{\left(\frac{a-b}{2}\right)}=\sin^2(2t)$$ which is the closed elliptic disk centered at the origin with foci at $\pm\sqrt{ab}$, major axis of length $a+b$, and minor axis of length $a-b$.
\end{proof}

\begin{lemma}\label{NRGP}
  Let $P$ and $Q$ be two orthogonal projections onto two subspaces $M$ and $N$ of $\comp$ in generic position. The numerical range of $PQ$ restricted to the $k$-th Jordan plane, $W(PQ|_{J_k})$, is the elliptic disk centered at $\frac{\lambda_k^2}{2}$ with foci at $0$ and $\lambda_k^2$,the eigenvalues of $PQ$, with major axis of length $\lambda_k$, with minor axis of length $\lambda_k \mu_k$.
  In other words,
    $$W(PQ|_{J_k})=\left\{x+iy \in \mathbb{C} : \frac{\left(x-\frac{\lambda_k^2}{2}\right)^2}{\frac{\lambda_k^2}{4}} + \frac{y^2}{\frac{\lambda_k^2 \mu_k^2}{4}} \leq 1 \right\}.$$
\end{lemma}

\begin{proof}
  Note that $PQ|_{J_k}=\langle v_k, u_k \rangle u_k \otimes v_k^*= (u_k \otimes u_k^*)\cdot (v_k \otimes v_k^*) = (P|_{J_k})( Q|_{J_k})$.
  Observe that the matrix representations of $P|_{J_k}$ and $Q|_{J_k}$ with respect to the orthonormal basis $\{u_k, s_k\}$ for $J_k$ are $P|_{J_k}=\begin{bmatrix}
    1 & 0 \\
    0 & 0
  \end{bmatrix}$ and $Q|_{J_k}=\begin{bmatrix}
    \lambda_k^2 & \lambda_k \mu_k \\
    \lambda_k \mu_k & \mu_k^2
  \end{bmatrix},$ whence $$PQ|_{J_k}= \begin{bmatrix}
    \lambda_k^2 & \lambda_k \mu_k \\
    0 & 0
  \end{bmatrix}.$$  Let $A=PQ|_{J_k}-\frac{\lambda_k^2}{2}I|_{J_k}$. Then $$A= \begin{bmatrix}
    \frac{\lambda_k^2}{2} & \lambda_k \mu_k \\
    0 & -\frac{\lambda_k^2}{2}
  \end{bmatrix}.$$
  Since $\langle u_k-t_k, u_k+t_k \rangle =0, \|u_k-t_k\|=\sqrt{\langle u_k-t_k,u_k-t_k \rangle} = \sqrt{2+2\mu_k}$, and $ \|u_k+t_k\|=\sqrt{\langle u_k+t_k,u_k+t_k \rangle} = \sqrt{2-2\mu_k}$, $$U=\begin{bmatrix}
    \frac{1}{\sqrt{2+2\mu_k}}(u_k-t_k) & \frac{1}{\sqrt{2-2\mu_k}}(u_k+t_k)
  \end{bmatrix}$$ is a unitary matrix. Transforming $A$ through $U$, we have $$U^*AU=\begin{bmatrix}
    0 & \frac{\lambda_k (1+\mu_k)}{2} \\
    \frac{\lambda_k (1-\mu_k)}{2} & 0
  \end{bmatrix}.$$ Note that $0<\frac{\lambda_k (1-\mu_k)}{2}<\frac{\lambda_k (1+\mu_k)}{2}$. It follows from Lemma \ref{NRcd} and Proposition \ref{NumRan} (\ref{ute}) that $W(U^*AU)=W(A)=W(PQ|_{J_k}-\frac{\lambda_k^2}{2}I|_{J_k})$ is the elliptic disk centered at the origin whose foci are $\pm\sqrt{\frac{\lambda_k^4}{4}}=\pm\frac{\lambda_k^2}{2}$, whose major axis has length $\frac{\lambda_k (1+\mu_k)}{2}+\frac{\lambda_k (1-\mu_k)}{2}=\lambda_k$, and whose minor axis has length $\frac{\lambda_k (1+\mu_k)}{2}-\frac{\lambda_k (1-\mu_k)}{2}=\lambda_k \mu_k$. Finally, the desired result comes from $W(PQ|_{J_k})=W(A+\frac{\lambda_k^2}{2}I_{J_k})$ and Proposition \ref{NumRan} (\ref{abi}).
\end{proof}

\begin{remark}
  Two end points of the major axis ,$\frac{\lambda_k^2 \pm \lambda_k}{2}$, in $W(PQ|_{J_k})$ are eigenvalues of $\frac{PQ+QP}{2}$, and the imaginary parts of the two end points of the minor axis, $\pm \frac{\lambda_k \mu_k}{2}$, are eigenvalues of $\frac{PQ-QP}{2i}$.
\end{remark}

Let $\conv(E)$ denote the convex hull of $E \subset \mathbb{C}$.
\begin{theorem}\label{NRGP1}
  If $P$ and $Q$ are orthogonal projections onto two subspaces $M$ and $N$ of $\mathbb{C}^{2p}$ that are in generic position, respectively,
     the numerical range of $PQ$, $W(PQ)$, is the convex hull of the union of all elliptical disks, $\cup_{k=1}^p W(PQ|_{J_k})$.
     In notation,

     \begin{align*}
       W(PQ)=& \conv (\cup_{k=1}^p W(PQ|_{J_k})) \\
            =& \conv \left(\cup_{k=1}^{p} \left\{x+iy \in \mathbb{C} : \frac{\left(x-\frac{\lambda_k^2}{2}\right)^2}{\frac{\lambda_k^2}{4}} + \frac{y^2}{\frac{\lambda_k^2 \mu_k^2}{4}} \leq 1 \right\}\right).
     \end{align*}
\end{theorem}
\begin{proof}
  For each $f \in J_k$, $\langle PQ|_{J_k}f,f \rangle = \langle P|_{J_k}Q|_{J_k}f,f \rangle  = \langle Q|_{J_k}f,P|_{J_k}f \rangle= \langle Qf, Pf \rangle=\langle PQf, f \rangle$, which implies $\conv (\cup_{k=1}^p W(PQ|_{J_k})) \subset W(PQ)$. On the other hand, let $f$ be a unit vector in $\mathbb{C}^{2p}$. Since $\mathbb{C}^{2p}=\oplus_{k=1}^p J_k$, we can write $f= c_1 f_1+ c_2 f_2+ \cdots + c_p f_p $ for some unit vectors $f_k \in J_k$ for $k=1,...,p$, where $\sum_{k=1}^{p} |c_k|^2 =1$.  \\
   Then  \begin{align*}
       \langle PQf, f \rangle =& \left\langle PQ\sum_{k=1}^p c_k f_k, \sum_{k=1}^p c_k f_k \right\rangle \\
                              =& \left\langle \sum_{k=1}^{p}c_k (PQ|_{J_k} f_k), \sum_{k=1}^p c_k f_k \right\rangle \\
                              =& \sum_{k=1}^{p}|c_k|^2 \langle PQ|_{J_k} f_k,  f_k \rangle \in \conv (\cup_{k=1}^p W(PQ|_{J_k})),
     \end{align*}
  so $\conv (\cup_{k=1}^p W(PQ|_{J_k})) \subset W(PQ)$. 
  
  Therefore, $$\conv (\cup_{k=1}^p W(PQ|_{J_k})) = W(PQ).$$
\end{proof}

Recall that in the decomposition of $\comp$ we let $R_M=M \cap R$ and $R_N=N \cap R$, where $R$ is the orthogonal complement of $(M \cap N) \oplus (M \cap N^{\perp}) \oplus (M^{\perp} \cap N) \oplus (M^{\perp} \cap N^{\perp})$. It is easy to see that $R_M$ and $R_N$ are in generic position as two subspaces of $R$. Observe that $P_M=P_{M\cap N}+P_{M \cap N^{\perp}}+P_{R_M}$ and $P_N=P_{M\cap N}+P_{M^{\perp} \cap N}+P_{R_N}$. Therefore, the product of orthogonal projections $P_M$ and $P_N$ can be given as follows:
 \begin{align*}
 P_MP_N =& (P_{M\cap N}+P_{M^{\perp} \cap N}+P_{R_N})(P_{M\cap N}+P_{M^{\perp} \cap N}+P_{R_N})\\
        =& P_{M \cap N} + P_{R_M}P_{R_N} = P_{M \cap N} \oplus P_{R_M}P_{R_N}. \\
\end{align*}
Now we proved the general case of Theorem \ref{NRGP1}.

\begin{cor}
  If $M \cap N \neq \{0\}$, then $W(PQ)$ is the convex hull of $\cup_{k=1}^n W(PQ|_{J_k})\cup [0,1].$ In other words,
        $$W(PQ)= \conv ([0,1] \cup (\cup_{k=1}^p W(PQ|_{J_k})).$$
\end{cor}

\end{document}